\documentclass[a4paper,oneside,12pt]{amsart}
\usepackage{bbm}
\usepackage{graphicx}
\usepackage[dvipsnames,usenames]{color}

\usepackage{amssymb,latexsym,amsmath,amsfonts,amscd}

\theoremstyle{plain}

\newtheorem{thm}{Theorem}[section]

\newtheorem{lem}[thm]{Lemma}

\theoremstyle{definition}

\newtheorem{rem}[thm]{Remark}

\numberwithin{equation}{section}

\def\cD{{\mathcal D}}
\def\cO{{\mathcal O}}
\def\bR{\mathbb{R}}
\def\cDA{{\cD(A)}}

\def\bC{{\mathbb C}}
\def\bD{{\mathbb D}}
\def\pt{\partial}
\newcommand\id{\operatorname{id}}
\newcommand\Bloch{\mathcal{B}}
\newcommand\cte{\mathbbm{1}}

\def\al{\alpha}
\def\be{\beta}
\def\ga{\gamma}

\def\Re{\operatorname{Re}}

\def\Hol{\operatorname{Hol}}

  \newtheorem*{main thm}{Main Theorem}

\def\beq{\begin{eqnarray}}
\def\eeq{\end{eqnarray}}
\def\beqa{\begin{eqnarray*}}
\def\eeqa{\end{eqnarray*}}
\def\wt{\widetilde}

\def\Om{\Omega}

\def\pt{\partial}

\def\de{\delta}
\def\ga{\gamma}
\def\beqn{\begin{equation}}
\def\eeqn{\end{equation}}

\def\cL{{\mathcal L}}
\def\mg#1{}

\renewcommand{\epsilon}{\varepsilon}
\renewcommand{\phi}{\varphi}

\newcommand\re{\operatorname{Re}}
\newcommand\im{\operatorname{Im}}

\begin{document}
\title[Generators of $C_0$ semigroups]{Generators of $C_0$-semigroups\\ of weighted composition operators}
\author{Eva A. Gallardo-Guti\'errez}
\address{E. A. Gallardo-Guti\'errez \newline Departamento de An\'alisis Matem\'atico y\newline
Matem\'atica Aplicada \newline
Facultad de CC. Matem\'aticas,
\newline Universidad Complutense de
Madrid, \newline
 Plaza de Ciencias N$^{\underbar{\Tiny o}}$ 3, 28040 Madrid,  Spain
 \newline
and Instituto de Ciencias Matem\'aticas ICMAT(CSIC-UAM-UC3M-UCM),
\newline Madrid,  Spain } \email{eva.gallardo@mat.ucm.es}
\thanks{First author is partially supported by Plan Nacional  I+D grant no. PID2019-105979GB-I00, Spain}

\author{Aristomenis G. Siskakis}
\address{A. G. Siskakis \newline Department of Mathematics\newline
Aristotle University of Thessaloniki \newline
54124 Thessalonıki, Greece}
\email{siskakis@math.auth.gr}

\author{Dmitry Yakubovich}
\address{D. V. Yakubovich\newline
Departamento de Matem\'aticas,\newline  Universidad Aut\'onoma de Madrid,\newline
Cantoblanco, 28049 Madrid, Spain\newline
and Instituto de Ciencias Matem\'aticas ICMAT (CSIC-UAM-UC3M-UCM), \newline Madrid, Spain.}
\email {dmitry.yakubovich@uam.es}
\thanks{Third author is partially supported by
Spanish Ministry of Science, Innovation and Universities (grant no. PGC2018-099124-B-I00).
\newline
The first and the third author also acknowledge support from ``Severo Ochoa Programme for Centres of Excellence in R\&D'' (CEX2019-000904-S) of the
Ministry of Economy and Competitiveness of Spain and by the European Regional Development
Fund and from the Spanish
National Research Council, through the ``Ayuda extraordinaria a
Centros de Excelencia Severo Ochoa'' (20205CEX001).
}

\subjclass[2010]{47B33 (primary), 30C45 (secondary)}

\keywords{Weighted composition operators, $C_0$-semigroups, Cocycles of Holomorphic Flows}

\date{\today}

\begin{abstract}
We prove that in a large class of Banach spaces of analytic functions in the unit disc $\mathbb{D}$
an (unbounded) operator $Af=G\cdot f'+g\cdot f$ with $G,\, g$ analytic in $\mathbb{D}$
generates a $C_0$-semigroup of weighted composition operators if and only
if it generates a $C_0$-semigroup. Particular instances of such spaces are the classical Hardy spaces.
Our result generalizes previous results in this context and it is related to
cocycles of flows of analytic functions on Banach spaces. Likewise, for a large class of non-separable Banach
 spaces $X $ of analytic functions in $\mathbb{D}$ contained in the Bloch space, we prove that no non-trivial
 holomorphic flow induces a $C_0$-semigroup of weighted composition operators on $X$. This generalizes
 previous results in ~\cite{BCDM}  and ~\cite{AJS} regarding $C_0$-semigroup of
 (unweighted) composition operators.
\end{abstract}

\maketitle



\section{Introduction}

A natural question in the study of semigroups of operators  is to characterize their generators, whenever they are defined.
Recently, in \cite{ACh} and \cite{Gallardo-Yakubovich} generators of semigroups of composition operators were
characterized in the context of a large variety of Banach spaces of analytic functions in the unit disc $\mathbb{D}$.  In
particular, if an (unbounded) operator $Af=G\cdot f'$ on the classical Hardy space $H^p$, $1\leq p<\infty$
generates a $C_0$-semigroup, this is a semigroup of composition operators. See also the previous
works \cite{AChP1}, \cite{AChP2} or \cite{ChalPart2016}.

\medskip

The study of semigroups of composition operators on various function spaces of analytic functions traces back to
the pioneering work of Berkson and Porta \cite{Ber-Por} in the late seventies. As a particular instance, they
gave a characterization of the generators of semigroups of composition operators acting on the classical Hardy
spaces  $H^p$ induced by \textit{holomorphic flows} of analytic self-maps of the unit disc $\mathbb{D}$
showing, in turns, that these semigroups are always strongly continuous. Later on, Cowen \cite{Cowen} and
Siskakis \cite{Siskakis} found applications of this theory to the study of Ces\`aro and other averaging
operators. Semigroups of composition operators acting on other holomorphic function spaces  as well as the
inducing holomorphic flows have been extensively studied in the last three  decades (we refer to the recent
monograph \cite{BCD} for more on flows).
\medskip

Nevertheless, in what refers to semigroups of weighted composition operators the panorama differs completely.
Seminal works in the context of Hardy spaces were  \cite{Sis2} and \cite{Konig} by the end of the eighties.
More recently, in 2012, the authors of \cite{Jafari-CAOT12} studied which strongly continuous semigroups of
operators on Banach spaces of analytic functions arise from holomorphic flows, where semigroups  of weighted
composition operators play an important role.
\medskip

Our work has two main results, which in some sense, complement each other. The first one
characterizes  generators of strongly continuous semigroup (i.e. $C_0$-semigroups) of weighted composition operators
acting on a large class of Banach spaces $X $ of analytic functions in the unit disc $\mathbb{D}$.
More precisely, in Section~\ref{Section 2} we will show that whenever the  polynomials are dense in the
Banach space $X $, an (unbounded) operator $A$ acting on $X $ like
$$
Af=G\cdot f'+g \cdot f,
$$
with $G$ and $g$ analytic in $\mathbb{D}$ generates a $C_0$-semigroup of weighted composition operators if and only if it
generates a $C_0$-semigroup. In order to prove it, we will make use of some of the ideas in \cite{Gallardo-Yakubovich},
along with the concept of cocycles of flows of analytic functions on Banach spaces.
\medskip

Our second result deals with classes of Banach spaces
$X $ where polynomials are no longer dense, namely, with
spaces satisfying
$$
H^\infty \subseteq X \subseteq \Bloch,
$$
where $\Bloch $ denotes the Bloch space and $H^\infty$ is the algebra of bounded analytic functions on
$\mathbb{D}$.

In particular, in Section~\ref{Section 3} we will prove that in such a case, a nontrivial holomorphic flow
cannot give rise to a non-trivial weighted composition $C_0$-semigroup. This extends previous
results in \cite{BCDM} and  \cite{AJS} regarding composition operators $C_0$-semigroups, that is, the unweighted case.

\section{Preliminaries}

\medskip

Recall that a one-parameter family $\Phi=\{\phi_t\}_{t\ge 0}$ of analytic self-maps of $\mathbb{D}$ is called
a \textit{holomorphic flow} (or \textit{holomorphic semiflow} by some authors) provided that it is a continuous
family that has a semigroup property with respect to composition, namely
\begin{enumerate}
\item[1)]
$\phi_0(z)=z$, $\forall z\in \bD$;

\item[2)] $\phi_{t+s}(z)=\phi_{t}\circ \phi_{s}(z)$, $\forall
t,s\ge0, \, \forall z\in \bD$;

\item[3)] For any $s\ge0$ and any $z\in\bD$,
$\lim_{t\to s}\phi_{t}(z)=\phi_{s}(z)$,
\end{enumerate}
(see \cite{Sh}, for instance). In the trivial case, $\varphi_t(z) =z$ for all $t\geq 0$. Otherwise, we say that $\Phi$ is nontrivial.
\smallskip

The infinitesimal generator of  $\Phi=\{\phi_t\}_{t\ge 0}$ is the function $G$, defined
by
\begin{equation*}
	G(z)=\lim_{t\to 0^+} \frac{\phi_t(z)-z}{t}= \frac{\partial \varphi_t(z)}{ \partial t}\arrowvert_{t=0}.
\end{equation*}
This limit exists uniformly on compact subsets of $\mathbb{D}$ and for each $z\in \mathbb{D}$, the mapping
$w:\, t\to \varphi_t(z)$ satisfies the differential equation
$$
\frac{d}{dt} w(t)= G(w(t))
$$
$\mbox{ with } w(0)=z$ (see \cite{Ber-Por}).

\medskip

Associated to the holomorphic flow $\Phi=\{\phi_t\}_{t\ge 0}$ is the family of composition operators $\{C_t \}_{t\geq 0}$, defined
on the space of analytic functions on $\mathbb{D}$ by
$$
C_t f= f\circ \phi_t.
$$
Clearly, $\{C_t \}_{t\geq 0}$ has the semigroup property:
\begin{enumerate}
\item[1)] $C_0=I$ (the identity operator);

\item[2)]$C_t C_s=C_{t+s}$ for all $t,s\ge 0$.
\end{enumerate}
Moreover, recall that if an operator  semigroup $\{T_t\}_{t\geq 0}$ acts on an (abstract) Banach space
$X $, then it is called \emph{strongly continuous} or \emph{$C_0$-semigroup}, if it satisfies
$$
\lim_{t\to 0^+} T_t f=f
$$
for any $f\in X $. As mentioned in the introduction, Berkson and Porta \cite{Ber-Por} proved that every holomorphic flow
always induces a strongly continuous semigroup on the classical Hardy spaces $H^p$ and extensions to other spaces of analytic
functions have been extensively studied in the last years.

Given a $C_0$-semigroup $\{T_t\}_{t\geq0}$ on a Banach space $X $, recall that its generator is the  closed and
densely defined linear operator $A$,
defined by
\[
Af=\lim_{t\to 0^+} \frac {T_t f-f} t
\]
with domain $\cDA=\{f\in X : \lim_{t\to 0^+} \frac {T_t f-f} t \enspace \text{exists}\}$. The semigroup is
determined uniquely by its generator.
\medskip

On the other hand, given a  holomorphic flow $\Phi=\{\phi_t\}_{t\ge 0}$ in $\mathbb{D}$,
a multiplicative \emph{cocycle} for $\Phi$ is a continuous complex-valued function
$m:[0,+\infty)\times \mathbb{D} \rightarrow \mathbb{C}$ such that
\begin{enumerate}
\item[1)] $m(t, \cdot )$ is analytic on $\mathbb{D}$ for all $t\geq 0$;

\item[2)]  $m(0,z)=1$, $\forall z\in \bD$;

\item[3)] $m(t+s,z)= m(s,z)\,  m(t, \phi_{s}(z))$, $\forall
t,s\ge0, \, \forall z\in \bD$.
\end{enumerate}
For the sake of simplicity, we will denote $m_t(z)=m(t,z)$ for all $t\geq 0$ and $z\in \mathbb{D}$.
\medskip

The third equation above is often called the \emph{cocycle identity} and it implies that $m_0(z)$ is either $1$ or $0$,
so the second equation is simply  a non-triviality condition. K\"onig \cite{Konig} investigated weighted holomorphic flows
on the unit disc and provided a characterization of smooth cocycles in the Hardy space. In addition, he proved that the cocycle
identity implies that $m$ is not vanishing (see \cite[Lemma 2.1]{Konig}).
\medskip

A particular instance of a cocycle for a holomorphic flow $\Phi=\{\phi_t\}_{t\ge 0}$ is a \emph{coboundary}, that is, a
continuous complex-valued function $m$ on $[0,+\infty)\times \mathbb{D}$ such that there exists holomorphic function
$\alpha$ on $\mathbb{D}$, non-vanishing except possibly at the common fixed point of $\phi_t$, satisfying

\begin{equation}\label{coboundary eq}
m(t,z)=\frac{\alpha(\varphi_t(z))}{\alpha(z)} \mbox{ for all } (t,z)\in [0,+\infty)\times \mathbb{D}.
\end{equation}
\medskip

Cocycles and, in particular, coboundaries arise naturally in the theory of weighted composition operators semigroups.
Observe that given a holomorphic flow $\Phi=\{\phi_t\}_{t\ge 0}$ and a cocycle $m$ for $\Phi$, the weighted
composition operators $\{W_t: t\geq 0\}$ defined  on the space of analytic functions on $\mathbb{D}$ by
\begin{equation}\label{eq weighted}
W_t f(z)= m_t(z)(C_t f)(z)=m_t(z) f(\varphi_t(z)), \qquad (z\in \mathbb{D})
\end{equation}
 form a semigroup. Indeed the converse also holds. Namely, \eqref{eq weighted} defines a
semigroup if and only if $m$ is a cocycle for $\Phi$. In addition, if $m$ is a coboundary then \eqref{coboundary
eq} easily yields that
$$
W_t= M_{\frac{1}{\alpha}} C_t\,  M_{\alpha}
$$
for every $t\geq 0$, where $M_{\alpha}$ is the operator of multiplication by $\alpha(z)$. This says that for
every $t\geq 0$, $W_t$ and $C_t$ are similar as operators on the space of all analytic functions on
$\mathbb{D}$.
\medskip

Throughout the rest of the manuscript, $X $ will be a Banach space of holomorphic
functions on $\mathbb{D}$ and $\cL(X )$ will denote the space of bounded
linear operators acting on $X $. The space consisting of
all analytic functions on $\bD$ is denoted
by $\Hol(\bD)$ while $\mathcal{O}(\overline \bD)$ stands for the set of all functions, holomorphic
on (a neighborhood of) the closed unit disc $\overline
\bD$ endowed with the usual topologies (regarding $\mathcal{O}(\overline \bD)$ we refer to
\cite[pp. 81]{Ber-Gay}, for instance).
We will always assume the following natural condition:
\newline
\\
$
(\star) \hspace*{6cm}  \cO(\overline \bD)\hookrightarrow{} X \hookrightarrow{} \Hol(\bD),
$
\newline
\\
being both embeddings continuous.
\medskip

In order to complete this preliminary section, we reproduce a theorem essentially  contained in \cite{Konig} for $H^p$
spaces (see also \cite{Sis2} for the case of coboundaries). The proof works as well for Banach spaces $X $ satisfying $(\star)$.

\begin{thm} \label{siskakis-konig-thm}
	Let $\{W_t\}_{t \ge 0}$ be a weighted composition $C_0$-semigroup acting on $X $
	defined by \eqref{eq weighted}, and let $G$ be the infinitesimal generator of $\{\varphi_t\}$.
	Then the following holds.
	
	\begin{enumerate}
		\item[(i)]
		The equation
		\begin{equation}\label{m-and-g}
		g(z)=\frac{\partial m_t(z)}{ \partial t}\arrowvert_{t=0},\qquad  \mbox{ for } z\in \mathbb{D}
		\end{equation}
		defines an analytic function in $\bD$, and $m_t$ is a cocycle for $\Phi$.
		Moreover, $m_t$ is expressed by
		\begin{equation}\label{weight expression}
		m_t(z)=\exp \left ( \int_0^t g(\varphi_s (z)) ds \right ), \qquad  t\ge0, \, z\in \mathbb{D}.
		\end{equation}
		
		\item[(ii)] The generator $A$ of the semigroup $\{W_t\}_{t\geq 0}$ and its domain are given by
		\begin{equation}\label{generator eq}
		Af=G\cdot f'+g \cdot f,
		\quad
		\cDA=\left \{ f\in X :\; G\cdot f'+g \cdot f \in X \right \}.
		\end{equation}
	\end{enumerate}
\end{thm}

\smallskip

Roughly speaking, our main result in Section \ref{Section 2} can be seen as the converse of this theorem.

\smallskip

Finally, we remark that whenever the cocycle $m$ for $\Phi$ is \emph{bounded}, that is, for every $t\geq 0$
the analytic function $m_t$ is bounded on $\mathbb{D}$, the weighted composition operator $W_t$ is a bounded
operator on the Hardy spaces $H^p$ for every $1\leq p\leq \infty$.
In the opposite direction, if $g$ is analytic with $\Re g(z)<M<\infty$ on $\mathbb{D}$, then
the above formulas induce a strongly continuous semigroup $\{W_t\}_{t\geq 0}$.

\section{$C_0$-semigroups of weighted composition operators: Generators} \label{Section 2}

In this section we will prove that in a large class of Banach spaces $X $ of analytic functions in
the unit disc $\mathbb{D}$, including Hardy spaces $H^p$ $(1\leq p<\infty)$, an (unbounded) operator $A$
 defined by \eqref{generator eq}
generates a $C_0$-semigroup of weighted composition operators if and only
if it generates a $C_0$-semigroup of operators.

\begin{thm} \label{thm-main}
Suppose $X $ is a Banach space of analytic functions on $\bD$,
satisfying $(\star)$.

Let $\{S_t\}_{t\geq 0}$ be a $C_0$-semigroup on $X $, whose generator $A$ is defined
by \eqref{generator eq},
where $G,\, g$ are analytic functions in $\bD$.
Then there exists a holomorphic flow
$\Phi=\{\varphi_t\}_{t\geq 0}$, whose infinitesimal generator is $G$ and a cocycle $m$ for $\Phi$,
satisfying  \eqref{m-and-g} and \eqref{weight expression}, such that $S_t$ is the weighted composition operator
\begin{equation*}
S_t f(z)= m_t(z) f(\varphi_t(z))
\end{equation*}
for every $t\geq 0, z\in \mathbb{D}$ and  $f\in X $.
\end{thm}

In order to prove the result, as it was mentioned in the Introduction, we will make use of some
 ideas in \cite{Gallardo-Yakubovich}.  
Note that the natural assumption $(\star)$ on $X $ was also
made in \cite{Gallardo-Yakubovich}.
\medskip

\begin{proof}
Assume the generator $A$ of the $C_0$-semigroup $\{S_t\}_{t\geq 0}$ is given
by $Af= G\cdot f'+g\cdot f$ where $G,\, g \in \Hol(\bD)$. We begin by posing the Cauchy
Problem as in \cite{Gallardo-Yakubovich}

$$
\eqno{(CP)}\qquad\qquad
\begin{cases}
\displaystyle \frac {\pt \phi_t(z)}{\pt t}=G(\phi_t(z)) \\
\phi_0(z)=z\qquad\qquad\qquad \big(z\in D(0,r)=\{z\in \bD: \; |z|<r\}\big),
\end{cases}
$$
where $r\in (0,1)$ is a fixed radius.  Hence, upon applying the theory of ordinary differential equations
in complex domain, there exists $t_0>0$
and an analytic solution $\{\phi_t(z)\}$ of (CP), defined for
$z\in D(0,r)$ and all complex $t$, $|t|<t_0$.
In particular, for these values of $z$ and $t$, $|\phi_t(z)|<1$.

Moreover, this solution is unique in the class of
smooth functions and has the semigroup property:
\begin{equation}\label{semigr-pr}
\phi_{t+s}(z)=\phi_{t}\circ\phi_{s}(z)
\end{equation}
whenever
$t,s,t+s\in(-t_0,t_0)$ and $z,\phi_{s}(z)\in D(0,r)$.

Choose $r'\in (0,r)$ such that
$\phi_{t}(z)\in D(0,r)$ for
$t\in (-t_0,t_0)$ and $z\in D(0,r')$
(here we may need to replace $t_0$ with a smaller positive number).
Then
\beqn
\label{phit}
\phi_{-t}\circ\phi_{t}(z)=z\qquad \text{if $t\in (-t_0,t_0)$ and $z\in D(0,r')$.}
\eeqn

First, we prove the following:

\

\noindent \underline{Claim 1:}
For any $f\in\cDA$
\beqn
\label{star}
\qquad
S_t f(z)=  \exp \left ( \int_0^t g(\varphi_s (z)) ds \right )\, (f\circ \phi_t)(z),
    \qquad z\in D(0,r), \, 0\le t< t_0.
\eeqn

\

\

\begin{proof}[Proof of Claim 1]
By (CP) we have $$\frac {\pt \phi_{-t}(z)}{\pt
t}=-G(\phi_{-t}(z)),$$
for any $z\in D(0,r)$ and any $t\in (0,t_0)$.  Let us fix $f\in \cDA$ and denote by
\[
f _t(z)=S_t f(z), \quad z\in\bD, \; t\ge0.
\]
Calculating the derivative of $f_t\circ\phi_{-t}(z)$ with respect to $t$  we obtain
\begin{eqnarray*}
\displaystyle
\frac {\pt (f_t \circ \phi_{-t})(z)}{\pt t}
&= & \displaystyle \frac{\pt f_t}{\pt t}\big(\phi_{-t}(z)\big) + \frac{\pt f_t}{\pt z}\big(\phi_{-t}(z)\big) \,
\frac{\pt \phi_{-t}(z)}{\pt t} \nonumber \\
\noalign{\medskip}
&= & \big(A  f_t\big) \big(\phi_{-t}(z)\big) - \frac{\pt f_t}{\pt z} (\phi_{-t}(z)) G (\phi_{-t}(z)) \nonumber \\
\noalign{\medskip}
&= & \displaystyle G (\phi_{-t}(z)) \frac{\pt f_t}{\pt z} (\phi_{-t}(z)) + g (\phi_{-t}(z)) f_t (\phi_{-t}(z))
- \frac{\pt f_t}{\pt z} (\phi_{-t}(z)) G (\phi_{-t}(z)),
\end{eqnarray*}
where the second line follows since $A$ is the generator of the semigroup,
 and hence $\displaystyle \frac {\pt f
_t}{\pt t} = Af _t$. Accordingly,
\begin{equation} \label{dif eq}
\frac {\pt (f_t \circ \phi_{-t})(z)}{\pt t}= g (\phi_{-t}(z)) f_t (\phi_{-t}(z)), \quad z\in D(0,r), \; 0\leq t<t_0.
\end{equation}
For $z\in D(0,r)$, let us write $Y(t)=(f_t \circ \phi_{-t})(z)$ and $\psi(t)=g (\phi_{-t}(z))$. Then \eqref{dif eq}
is the differential equation
$$
Y'(t)= \psi(t)\, Y(t),
$$
so
$$
Y(t)\, \exp \left (- \int_0^t \psi(u) \,du \right ) = C_z,
$$
where $C_z$ is a constant (depending on $z$). Thus, for each $z\in  D(0,r)$
$$
(f_t \circ \phi_{-t})(z)= C_z \, \exp \left (\int_0^t g (\phi_{-u}(z)) \,du \right ).
$$
For $t=0$ we have $C_z=f_0( \phi_{0} (z))= f(z)$ and therefore
\begin{equation}\label{eq expression}
(f_t \circ \phi_{-t})(z)=  \exp \left (\int_0^t g (\phi_{-u}(z)) \,du \right )\, f(z)
\end{equation}
for $z\in D(0,r)$.
Replacing $z$ by $\varphi_t(z)$ in \eqref{eq expression}
we deduce by means of analyticity that
\begin{eqnarray*}
f_t(z)&=& \exp \left (\int_0^t g (\phi_{t-u}(z)) \,du \right )\, f(\phi_t(z))\\
&=& \exp \left (\int_0^t g (\phi_{s}(z)) \,ds \right )\, f(\phi_t(z)),
\end{eqnarray*}
for all $z\in D(0,r)$. This proves  Claim 1.
\end{proof}
\medskip

Now, let $f\in X $. There exists a sequence $\{f_n\}\subset \cDA$ such that $f_n\to f$ in $X $
as $n\to \infty$. Hence, $f_n\to f$ uniformly on compact subsets of $\mathbb{D}$ as $n\to \infty$, and therefore
$$
\lim_{n\to \infty} S_t f_n(z)= S_t f(z)
$$
for each $z\in \mathbb{D}$ and $t\geq 0$. Accordingly, by Claim 1, for any $f\in X $
\begin{equation}\label{eq general function}
S_t f(z)= \exp \left ( \int_0^t g(\varphi_s (z)) ds \right )\, (f\circ \phi_t)(z), \qquad z\in D(0,r), \, 0\le t< t_0.
\end{equation}

Now, upon applying \eqref{eq general function} to the constant function $\cte(z)=1$, we deduce
\begin{equation}\label{local cocycle}
S_t (\cte)(z)= \exp \left ( \int_0^t g(\varphi_s (z)) ds \right )\, 
\end{equation}
for  $z\in D(0,r), \, 0\le t< t_0.$ \textit{Define}
\begin{equation}\label{global cocycle}
m_t(z):=S_t (\cte)(z)  \quad z\in \bD, \, t\ge 0.
\end{equation}
Clearly, this definition agrees with \eqref{local cocycle} if $0\le t<t_0$ and
$|z|<r$.  Therefore $m_t(z)\ne 0$ for
$0\le t< t_0$ and $|z|<r$. Moreover, $m_t\in X $ and in
particular, $m_t$ is analytic on $\bD$, for any $t\ge 0$.

Now apply  \eqref{eq general function} to the identity  function $\id$, $\id(z)\equiv z$. We get
\begin{equation*}
S_t (\id)(z)=\exp \left ( \int_0^t g(\varphi_s (z)) ds \right )\, \phi_t(z) \in X
\end{equation*}
for  $z\in D(0,r)$ and $0\le t< t_0$. Therefore
\begin{equation}\label{phit local}
\phi_t(z)= \frac 1 {m_t(z)}\, S_t (\id)(z).
\end{equation}
for every $z$, $|z|<r$ and $0\le t< t_0$.
Hence for any such $t$,
$\phi_t$ continues to a meromorphic function in $\mathbb{D}$.

\medskip

Analogously, applying \eqref{eq general function} to the functions $z^n$ for $n\geq 2$,
it follows for every $a\in \mathbb{D}$
\begin{equation}\label{phitn local}
(\phi_t(a))^n= \frac 1 {m_t(z)}\, S_t (z^n)(a),
\end{equation}
for $0\le t< t_0$. If for some
$0\le t< t_0$, $\phi_t$ had a pole in a point $a_0\in \mathbb{D}$,
then $\phi_t^n$ would have at least $n$th order pole at $a_0$ for
any $n\ge 1$, which contradicts  \eqref{phitn local}. Therefore
$\phi_t$ is analytic in $\mathbb{D}$ for any $t$, $0\le t< t_0$.

Now, a similar argument to that in \cite[Claim 2]{Gallardo-Yakubovich} along
with \eqref{phitn local} yields that there exists a positive $t_1 \le t_0$ such that
$|\phi_t(z)|< 1$ for all $z\in\bD$ and all $0\le t< t_1$.
\medskip

Finally, observe that from \eqref{semigr-pr} we get
\[
\phi_s\circ\phi_t(z)=\phi_{s+t}(z)
\]
for $s, t\ge 0$, $s+t<t_1$ and all $z\in\bD$.
Since the family $\{\phi_t\}$ satisfies (CP), it follows that
it can be continued to a holomorphic flow,
defined on $[0,+\infty)\times \bD$ (see \cite[Proposition 3.3.1]{Sh}).
Now we can assert that the right hand side of \eqref{eq general function} defines a semigroup
on the space $\operatorname{Hol}(\bD)$. Since
this semigroup coincides with the semigroup $\{S_t\}$ for $0\le t<t_0$,
if follows that \eqref{eq general function} holds for all $t>0$ and all $z\in \bD$.
In particular, by applying \eqref{eq general function} to the function $\cte$,
we get \eqref{weight expression}.

The rest of the proof runs as in \cite{Gallardo-Yakubovich}, which yields the statement of the Main Theorem.
\end{proof}

\begin{rem}
Note that, by hypotheses, the weighted composition operators act boundedly on $X $. Nevertheless,
we point out that even in the Hardy spaces $H^p$ it is unknown how to characterize bounded weighted composition
operators in terms of the inducing symbols (see \cite{GGKP} for related discussions).
\end{rem}

\begin{rem}
It is to be noticed that our results have consequences for spaces of analytic functions on domains other than the
disc. More precisely let $\Om\subsetneq\bC$ be a simply connected  domain and suppose $X$ is a Banach space
consisting of functions analytic on $\Om$, which satisfies the analogous condition $(\star)$ for $\Om$. Suppose
$\{S_t\}_{t\geq 0}$ is a $C_0$-semigroup acting on $X$, with infinitesimal generator of the form
$$
Af= G\cdot f' +g\cdot f, \quad \mathcal{D}(A)=\{f\in X:  G\cdot f' +g\cdot f\in X\}
$$
where $G$ and $g$ are analytic on $\Om$. Let $h:\bD\to \Om$ is a fixed conformal map onto $\Om$. Consider the Banach space
$X(\bD)=\{F=f\circ h: f\in X\}$ with the transferred norm $\|F\|_{X(\bD)}=\|f\|_{X}$ and assume that
$X(\bD)$ satisfies $(\star)$. The composition operator  $C_h(f)=f\circ h$ is an onto isometry between $X$ and $X(\bD)$,
with $C_h^{-1}=C_{h^{-1}}$. It is then clear that the family of operators $\{T_t\}_{t\geq 0}$,
$$
T_t = C_h\circ S_t\circ C_{h^{-1}}
$$
is a $C_0$-semigroup on $X(\bD)$. Its infinitesimal generator $\Gamma$, for $F\in \mathcal{D}(\Gamma)$,  is
\begin{align*}
\Gamma F&= \lim_{t\to 0^{+}}\frac{T_tF-F}{t}=\lim_{t\to 0^{+}}\frac{C_h\circ S_t\circ C_{h^{-1}}(F)-F}{t}\\
&=\lim_{t\to 0^{+}}\frac{ S_tf\circ h-f\circ h }{t}=(Af)\circ h, \quad \text{where} \,f=F\circ h^{-1}\in X,\\
&= G\circ h\cdot  f'\circ h+ g\circ h\cdot f\circ h\\
&= \frac{1}{h'} G\circ h\cdot F' + g\circ h\cdot F\\
&= G_1\cdot F' + g_1\cdot F,
\end{align*}
where $G_1(z) = \frac{1}{h'(z)} G(h(z))$ and $g_1(z) =g(h(z))$ are analytic in $\bD$. Thus by Theorem
\ref{thm-main} there is a holomorphic flow $\{\phi_t\}$ and a cocycle $m$ for $\{\phi_t\}$  such that
$$
T_tF(z)=m_t(z)F(\phi_t(z)), \quad t\geq 0, \,  z\in \bD, \,  F\in X(\bD).
$$
So we deduce that
\begin{align*}
S_t(f)(z) &=( C_{h^{-1}}\circ T_t\circ C_h)(f)(z)\\
&=m_t(h^{-1}(z)) ( f\circ h\circ\phi_t\circ h^{-1})(z)\\
&= \mu_t(z)f(\psi_t(z))
\end{align*}
for $f\in X$ and $z\in \Om$ where $\psi_t =h\circ\phi_t\circ h^{-1}$ is a holomorphic flow in $\Om$ and
$\mu_t(z)=m_t(h^{-1}(z))$ is a cocycle in $\Om$ for $\{\psi_t\}$.

More generally,
let $h$ is as before and let $\tau$ be a zero-free holomorphic function on $\bD$.
Then the same observation applies if $W_{h,\tau}$ is the weighted composition operator
$W_{h,\tau}(f)=\tau\cdot(f\circ h)$, and
the space $X(\bD)$, defined as $X(\bD)=\{W_{h,\tau} f: \;  f\in X\}$, satisfies $(\star)$. This remark can be applied, for instance, to
Hardy and to Smirnov spaces of bounded and unbounded domains (see ~\cite{Duren-book} for the corresponding definitions).
\end{rem}

\section{Absence of non-trivial  $C_0$-semigroups of weighted
composition\\ operators in non-separable Banach function spaces
}\label{Section 3}

As mentioned in the introduction, in this section we deal with Banach spaces $X$
of analytic functions on $\mathbb{D}$ such that
\begin{equation}\label{containment}
H^\infty \subseteq X \subseteq \Bloch.
\end{equation}
 Recall that  the Bloch space $\Bloch$ is a Banach space endowed with the norm
\begin{equation}\label{Bloch norm}
\|f\|_{\Bloch}= |f(0)|+ \sup_{z\in \mathbb{D}} |f'(z)| (1-|z|^2), \qquad (f\in \Bloch).
\end{equation}
It will be showed below that such space $X$ is always non-separable (see Remark \ref{non-separability}).

In \cite{BCDM}, the authors showed that no non-trivial holomorphic flow $\Phi=\{\varphi_t\}_{t\geq 0}$
induces a $C_0$-semigroup of composition operators on $H^{\infty}$ or $\Bloch$  by considering an argument
involving the Dunford-Pettis property. More recently,  Anderson, Jovovic and Smith ~\cite{AJS} proved that the
same holds for the space BMOA as well as for any Banach space $X$ satisfying \eqref{containment} by means of
a  geometric function theory  argument. Based on it, we will prove the following

\begin{thm} \label{main-thm-2}
Let $X $ be a Banach space of analytic functions on $\mathbb{D}$ such that
$H^\infty \subseteq X \subseteq \Bloch$.  Assume that $\Phi=\{\phi_t\}_{t\geq 0}$ is a nontrivial
 holomorphic flow. Let $\{W_t\}_{t\geq 0}$ be any weighted composition semigroup induced by $\Phi$, namely,
 suppose there exists a cocycle $m$ for $\Phi$ such that
\begin{equation}
W_t f(z)=m_t(z) f(\varphi_t(z)), \qquad (z\in \mathbb{D}),
\end{equation}
for any $f\in X $. Then $\{W_t\}_{t\geq 0}$ is no longer a strongly continuous semigroup in $X$.
\end{thm}

	We remark that for any reasonable Banach space $X$ of analytic functions in $\bD$, the Banach
 space $[\phi_t, X]\subset X$ is defined; it is the maximal
	subspace of $X$, on which the unweighted composition semigroup $\{C_t\}$ is strongly continuous.
	The recent work
    \cite{ChalmDask2021arx} is devoted to a description of
    flows $\{\phi_t\}$ such that the spaces $[\phi_t,\Bloch]$ and $[\phi_t, BMOA]$ are the minimal ones, that is,
    $[\phi_t,\Bloch]=\Bloch_0$ and $[\phi_t, BMOA]=VMOA$.

In the proof of Theorem \ref{main-thm-2}, we will make use the following lemma, which can be extracted from
\cite{AJS}, see also \cite{GMS}.

\begin{lem}
	\label{lem1}
	For any holomorphic flow $\{\phi_t\}_{t\geq 0}$, one of the two statements holds:
	\begin{enumerate}
	\item There is a point $\ga_0\in\pt\bD$ such that
	\[
	\exists \lim_{r\to 1^-} \phi_t(r\ga_0):=\phi_t(\ga_0), \quad t>0,
	\]
	and $\phi_t(\ga_0)\in \bD$ for any $t>0$;
	
	\item $\phi_t$ is an automorphism of $\bD$ onto itself for any $t\ge 0$.
	
\end{enumerate}
\end{lem}

This is implicit in the proof of Theorem 3.1 in \cite{AJS}, which is based on the K\"onig model of $\{\phi_t\}_{t\geq 0}$.
Namely, there is a Riemann map $h$ of $\bD$ onto a simply connected domain $\Om$ with $h(0)=0$ and a
 constant $c$, $\re c\ge 0$, such that either
\begin{equation}
\phi_t(z)= h^{-1}\big(e^{-ct}h(z)\big), \quad z\in \bD, \, t\ge0
\end{equation}
or
\begin{equation}
\phi_t(z)= h^{-1}\big(h(z)+ct\big), \quad z\in \bD, \, t\ge0.
\end{equation}
In the first case $\Om$ has to be spirallike,  $we^{-ct}\in \Omega$ for each $w\in \Omega$ and $t\geq 0$. In the
second one, $\Om$ is close to convex, $w+ct \in \Om$ for each $w\in \Omega$ and $t\geq 0$. We refer to \cite{Po75}
 or \cite{Sh} for more on the subject.

\medskip

A key ingredient in the proof will be a result which allows us to estimate the derivative of an infinite interpolating
 Blaschke products in pseudohyperbolic discs. Of particular help will be Lemma 3.5 from the paper \cite{GPV}
 by Girela, Pel\'aez and Vukotic. In order to formulate it, we recall that
given
$\theta\in\mathbb{R}$ and
a sequence of (not necessarily distinct) points
$\{z_k\}$ in $\mathbb{D}$
satisfying the \emph{Blaschke
condition} $$\sum_{k=1}^{\infty} (1-|z_k|)<\infty,$$ the infinite
product
$$
B(z) = e^{i\theta}\,\prod_{k=1}^{\infty} \frac{|z_k|}{z_k}
\frac{z_k-z}{1-\overline{z_k}z}
$$
converges uniformly on compact subsets of
$\mathbb{D}$.

For each $k\geq 1$, we will denote by $b_{z_k}$
the factor $\frac{|z_k|}{z_k}
\frac{z_k-z}{1-\overline{z_k}z}$.
Whenever $z_k=0$, the above expression
$|z_k|/z_k$ is taken to be equal to $-1$, so that
$b_{z_k}(z)=z$ in this case.
The holomorphic function $B$ is called the \emph{Blaschke product with zero sequence $\{z_k\}$}. It satisfies
\begin{itemize}
\item[i)] $B$ vanishes precisely at the points $\{z_k\}$, with
the corresponding multiplicities (that is, $\{z_k\}$ is the
zero-sequence of $B$), and

\item[ii)] $|B(z)|<1$ for every $z\in \mathbb{D}$.

\item[iii)] $\lim_{r\to 1^{-}} |B(re^{i\theta})|=1$ almost everywhere on the boundary $\partial
\mathbb{D}$.
\end{itemize}

Recall that a sequence $\{z_n\}$ is said to be interpolating if
for any bounded sequence of complex numbers $\{w_n\}$ there exists a function $f \in H^\infty$
such that $f(z_n) = w_n$ for all $n$.
A Blaschke product is interpolating if its zero sequence is interpolating (in this case, all $z_n$ are distinct).
It is known that $B$ is interpolating if and only if
there is a constant $\de>0$ such that
$|B_n(z_n)|>\de>0$ for all $n$, where $B_n=B/b_{z_n}$.
Whenever $1-|z_{n+1}|\le \al (1-|z_n|)$ for all $n$,
where $\al<1$ is a constant, the sequence $\{z_n\}$ is interpolating.

For these and other properties of Blaschke products, we refer  to Garnett's book \cite{Ga}.

In the next statement, $\Delta(a,r)$ will denote the pseudohyperbolic disc
$$\Delta(a,r)=\big\{z\in \bD: \quad \rho(z,a)<r \big\},$$
where $\rho(z,w)$ stands for the pseudohyperbolic
distance between two points $w,\, z$ in $\mathbb{D}$, namely,
$\rho(z,w)=\left |\frac{w-z}{1-\overline{w}z} \right |.$

\begin{lem}
\label{lem-GPV}
Suppose $B$ is an infinite  Blaschke product in $\bD$. Let $\{a_n\}$ be a part of its zeros. Suppose
that it is an interpolating sequence and that
\[
|(B/b_{a_n})(a_n)|\ge \de>0.
\]

Then there are positive constants $\al$ and $\be$, which depend only on
$\de$, such that  the pseudo-discs $\Delta(a_n,\al)$, $n\ge 1$ are pairwise disjoint and
\[
|B'(z)|\ge \frac \be {1-|a_n|}, \quad z\in \Delta(a_n, \al).
\]
\end{lem}

Though the statement of Lemma \ref{lem-GPV} is formally more general than the one of Lemma 3.5 in \cite{GPV},
the proof is obtained by means of the same argument.

\medskip

We are now in position to prove Theorem \ref{main-thm-2}.

\medskip

\begin{proof}[Proof of Theorem \ref{main-thm-2}]
By the Closed Graph Theorem, the inclusion $X\subseteq \Bloch$ is continuous.
Hence it suffices to prove that there exist a function $f\in H^\infty \subseteq X$
and a constant $\de>0$ such that
\begin{equation}
\label{limsup}
\limsup_{t\to 0^+}\|W_tf-f\|_{\Bloch}\ge \de.
\end{equation}

Assume first that in Lemma \ref{lem1}, case (1) holds. Without loss of generality, we can assume that $\ga_0=1$.
We are going to choose an increasing sequence $r_n\nearrow 1$ and a decreasing one $t_n\searrow 0$
such that $\{r_n\}\cup \{\phi_{t_n}(r_n)\}$
is an interpolating sequence in $\bD$ (in particular, the sequence
$\{r_n\}$ is disjoint from $\{\phi_{t_n}(r_n)\}$).
Moreover, we require that
\[
|\phi_{t_1}(r_1)|<r_1<\dots
<\big|\phi_{t_{n-1}}(r_{n-1})\big|<r_{n-1}
<\big|\phi_{t_n}(r_n)\big|<r_n<\dots
\]
and
\beqn
\label{geom-prop}
1-r_n<\frac 12 \big(1-|\phi_{t_n}(r_n)|\big) < \frac 14 (1-r_{n-1})
\eeqn
for all $n\ge1$. (For $n=1$, we only require the first inequality.)
This will be done by an inductive construction.

Take $t_1>0$ arbitrarily. Since $|\phi_{t_1}(1)|=\lim_{r\to 1^{-}}|\phi_{t_1}(r)|<1$, one can
choose $r_1<1$ so that
\[
1-r_1<\frac 12 \big(1-|\phi_{t_1}(r_1)|\big),
\]
and in particular, $|\phi_{t_1}(r_1)|<r_1$.
This is the first step of the construction.

Suppose $t_{n-1}$ and $r_{n-1}$ have been constructed.
Choose $t_n\in (0, \frac 12 t_{n-1})$ so that
\[
0<1-|\phi_{t_n}(1)|< \frac 12 (1-r_{n-1}).
\]
Then for any $r<1$ sufficiently close to $1$,
\[
1-|\phi_{t_n}(r)|< \frac 12 (1-r_{n-1}).
\]
Hence $r_n<1$
can be chosen in such way that
\[
1-r_n<\frac 12\big(1-|\phi_{t_n}(r_n)|\big)< \frac 14 (1-r_{n-1}).
\]
This defines sequences
$r_n\nearrow 1$ and $t_n\searrow 0$ with all desired properties.

We put
\[
f=\wt B \widehat B^2
\]
where the Blaschke products
$\wt B$, $\widehat B$ correspond to sequences of zeros
$\{r_n\}$, $\{\phi_{t_n}(r_n)\}$, respectively.
	By \eqref{geom-prop}, $\wt B \widehat B$ is an interpolating Blaschke product.

Having in mind \eqref{Bloch norm}, we note that
\begin{eqnarray*}
\|W_t f-f\|_{\Bloch }&\geq & \sup_{z\in \mathbb{D}} \big |(m_t (f \circ \varphi_t))'(z)- f'(z)\big | (1-|z|^2)\\
&\geq &  \sup_{0\le r<1}
\big|
m_t(r)\cdot f'\big(\phi_t(r)\big)\phi_t'(r)+m_t'(r)f\big(\phi_t(r)\big)-f'(r)
\big|(1-r).
\end{eqnarray*}

Since $f\big(\phi_{t_n}(r_n)\big)=f'\big(\phi_{t_n}(r_n)\big)=0$, by substituting $r=r_n$, $t=t_n$ in the above
inequality, we get
\[
\|W_{t_n}f-f\|_{\Bloch }\ge |f'(r_n)|(1-r_n)\ge \de>0,
\]
for all $n$ (we applied Lemma \ref{lem-GPV} to the Blaschke product $f$ and to $a_n=r_n$). This gives \eqref{limsup}.

\medskip

Now assume Case (2) in the structure Lemma~\ref{lem1}.
Then $\phi_t$ for $t>0$ is a disc automorphism, being either elliptic, parabolic or hyperbolic, and
this classification is the same for all $t$ (see \cite{BCD}, for instance).

\medskip

Hence, the flow $\{\varphi_t\}_{t\geq 0}$ can be continued to a holomorphic flow $\{\phi_t: \, t\in\bR\}$.
In the parabolic case, we will assume
(without loss of generality) that $1$ is not the Denjoy-Wolff point
of $\{\phi_t: t\ge 0\}$. In the hyperbolic case, we will assume that $1$ is not the Denjoy-Wolff point
neither of $\{\phi_t: t\ge 0\}$ nor of $\{\phi_{-t}: t\ge 0\}$.

Then, in all three cases, for a small $t_0>0$, the curves $\eta_r: [0,t_0]\to \bD$,
$$
\eta_r(t):=\phi_t(r),
$$
tend as $r\to 1^-$ in the $C^1$ metric to the curve $\eta_1$. This latter curve is a one-to-one smooth
parametrization of a small subarc
of $\pt\bD$, one of whose endpoints is $1$.

Put $r_n=1-2^{-n}$ for $n\ge N_0$, where $N_0$ is large enough. If the curve $\eta_1$ is contained in
$\{\im z\ge 0\}$, then there exist $t_n\searrow 0$ such that $\arg (\phi_{t_n}(r_n)-1)=3\pi/4$, $n\ge N_0$.
If $\eta_1$ is contained in $\{\im z\le 0\}$, then one can choose $t_n\searrow 0$ so
that $\arg(\phi_{t_n}(r_n)-1)=-3\pi/4$, $n\ge N_0$. In both cases,
\[
\frac{1-\re \phi_{t_n}(r_n)} {2^{-n}}\to 1\quad \text{ as }n\to\infty.
\]

It follows that
the sequence
$\{r_n\}\cup \big\{\phi_{t_n}(r_n)\big\}$ is contained in a Stoltz angle
with vertex at $1$ and is separated in the pseudohyperbolic metric $\rho(\cdot, \cdot)$.
Thus, this sequence is interpolating (see \cite{Nik}, Lecture VII) and the rest of the proof runs as in the Case~(1).
\end{proof}

\begin{rem}\label{non-separability}
	Any space $X$ satisfying \eqref{containment} is not separable. Indeed,
	 let $B$ be an interpolating Blaschke product with (infinitely many) zeros on the radius $[0,1)$, for instance, at the points $1-2^{-n}$. Put
$B_t(z)=B(\exp(-it)*z)$. Then each Blaschke product $B_t$ belongs to $X$. Using Lemma ~\ref{lem-GPV}, we see that the distances
	 $\|B_t-B_s\|_{\Bloch}$, where $t,s \in [0, 2\pi)$ and $t \ne s$, are uniformly bounded from below by some constant $\epsilon >0$ which implies our assertion. 
\end{rem}

\end{document}